\documentclass[11pt,reqno]{amsart}
\usepackage[T1]{fontenc}
\usepackage{lmodern}
\usepackage{amsmath,amssymb,mathtools}
\usepackage{microtype}
\usepackage[a4paper,margin=30mm]{geometry}
\usepackage[hidelinks]{hyperref}
\newtheorem{theorem}{Theorem}[section]
\newtheorem{proposition}[theorem]{Proposition}
\newtheorem{lemma}[theorem]{Lemma}
\newtheorem{corollary}[theorem]{Corollary}
\theoremstyle{remark}
\newtheorem{remark}[theorem]{Remark}
\newcommand{\N}{\mathbb N}
\newcommand{\Nzero}{\mathbb N_0}
\newcommand{\Z}{\mathbb Z}
\newcommand{\R}{\mathbb R}
\newcommand{\C}{\mathbb C}
\newcommand{\T}{\mathbb T}
\newcommand{\E}{\mathbb E}
\newcommand{\cP}{\mathcal P}
\newcommand{\cB}{\mathcal B}
\newcommand{\ee}[1]{\mathrm e(#1)}
\DeclareMathOperator{\supp}{supp}
\title[Non-recurrence along multiplicative semigroups]
{Weak mixing and non-recurrence along finitely generated multiplicative semigroups of $\mathbb N$}

\author{Zhuowei Liu}
\date{\today}
\address[Z. Liu]{School of Mathematics (Zhuhai), Sun Yat-sen University,
	Zhuhai, Guangdong, 519000, P.R. China}
\email{liuzhw55@mail2.sysu.edu.cn}
\subjclass[2020]{Primary 37A25; Secondary 37A44, 60G15}
\keywords{Weak mixing, non-recurrence, multiplicative semigroups, Gaussian shifts}
\begin{document}
\begin{abstract}
Let $\mathcal P$ be a finite nonempty set of primes, and let $S_{\mathcal P}$
be the multiplicative semigroup that they generate. We prove that
there are an invertible weakly mixing measure-preserving
transformation $T$ and a set $A$ of positive measure such that
$A\cap T^{-r}A=\varnothing$ for every $r\in S_{\mathcal P}$.
This gives a negative answer to Bergelson's 1996 question concerning
recurrence along $\{2^m3^n:m,n\geq1\}$, and applies to every
finitely generated multiplicative semigroup of positive integers.
More generally, we give a sufficient condition for weakly mixing
non-recurrence based on divisibility properties and separation
by finitely many rotations of the circle.
This criterion also applies to an infinitely generated
multiplicative semigroup not contained in any $S_{\mathcal P}$
with $\mathcal P$ finite.
\end{abstract}
\maketitle

\section{Introduction}
Poincar\'e's recurrence theorem guarantees that a set of positive
measure in a measure-preserving system returns to itself.
One may ask whether the return time can be required to belong to a
prescribed set of integers. Write $\N=\{1,2,\ldots\}$ and
$\Nzero=\{0,1,2,\ldots\}$. A set $R\subset\N$ is a set of measurable
recurrence if, for every invertible
measure-preserving system $(X,\cB,\mu,T)$ and every $A\in\cB$
with $\mu(A)>0$, there is an $r\in R$ such that
$\mu(A\cap T^{-r}A)>0$.

Consider the multiplicative semigroup
$R_{2,3}=\{2^m3^n:m,n\geq1\}$.
Since none of its elements is divisible by $5$, rotation by $1$ on
$\Z/5\Z$ shows that $R_{2,3}$ is not a set of measurable recurrence.
However, weakly mixing systems have no nontrivial measurable
rotation factors, so this example does not answer the question
for weakly mixing systems. In 1996, Bergelson~\cite[Question~3]{Bergelson} asked whether
every invertible weakly mixing system and every set $A$ of positive
measure satisfy
\[
  \mu(A\cap T^{-2^m3^n}A)>0
  \qquad\text{for some }m,n\geq1.
\]

Frantzikinakis, Host, and Kra recorded this question as open
in their 2025 paper~\cite[Section~3.2]{FHK}.\footnote{%
	After the first version of this paper was submitted to arXiv,
	Bryna Kra kindly shared with the author earlier unpublished
	notes giving a shorter proof of weakly mixing non-recurrence
	for $\{2^n3^m:n,m\ge1\}$, based on
	Grivaux's construction~\cite{Grivaux}.
	She also pointed out that the argument extends to
	arbitrary finitely generated multiplicative semigroups.
	The proof in the present paper was developed before
	the author became aware of these notes.}
By modifying a construction of Maucourant~\cite[Lemma~4.2]{Maucourant},
they showed that $R_{2,3}$ can fail to be recurrent even for a totally
ergodic torus rotation. They also proved that the larger set
$\{k!2^m3^n:k,m,n\geq1\}$ is Bohr recurrent, that is, recurrent
for every finite-dimensional torus rotation \cite[Theorem~1]{FHK}.
We give a negative answer to Bergelson's question.
The result holds for multiplicative semigroups generated
by any finite nonempty set of primes.

\begin{theorem}\label{thm:main}
Let $\cP=\{p_1,\ldots,p_d\}$ be a finite nonempty set of primes and put
\[
  S_{\cP}
  =\left\{p_1^{n_1}\cdots p_d^{n_d}:
                 n_1,\ldots,n_d\in\Nzero\right\}.
\]
There exist a standard
nonatomic probability space $(X,\cB,\mu)$, an invertible weakly mixing
probability-preserving transformation $T:X\to X$, and a Borel set
$A\subset X$ such that
\[
  \mu(A)>0
  \quad\text{and}\quad
  A\cap T^{-r}A=\varnothing
  \qquad\text{for every }r\in S_{\cP}.
\]
The system can be chosen to be a finite product of two-sided 
complex Gaussian shifts.
\end{theorem}

Every finitely generated multiplicative semigroup of positive integers
is contained in some $S_{\cP}$, since only finitely many primes divide
its generators. Thus the theorem applies to all such semigroups,
including the time $1$. Taking $\cP=\{2,3\}$ answers Bergelson's question.
Coding the set $A$ also gives a compact topologically weakly mixing
subshift with a nonempty clopen set having no return at any time in
$S_{\cP}$; see Proposition~\ref{prop:symbolic}.

Our argument also applies to sets involving
infinitely many distinct primes.
Write $\T=\R/\Z$ and $\ee{t}=\exp(2\pi i t)$. A set $R\subset\N$ is Bohr recurrent if, for every $m\ge1$,
$\alpha_1,\ldots,\alpha_m\in\T$, and $\varepsilon>0$,
there exists $r\in R$ such that
$\max_{1\le a\le m}|\ee{r\alpha_a}-1|<\varepsilon.$ For a nonempty set $H\subset\N$, $\gcd H$ denotes the
greatest positive integer dividing every member of $H$.
We first prove the following general theorem.

\begin{theorem}\label{thm:general}
Let $R\subset\N$ be of the form
\[
R=F\cup E_1\cup\cdots\cup E_d,
\]
where $d\ge1$, $F$ is finite, and
$E_1,\ldots,E_d$ are infinite. Suppose that
\begin{equation}\label{eq:tail-gcd}
  \lim_{N\to\infty}\gcd\{r\in E_i:r\geq N\}=\infty
  \qquad(1\leq i\leq d).
\end{equation}
Suppose also that there are $\alpha_1,\ldots,\alpha_m\in\T$, with
$m\geq1$, and $\delta>0$ such that
\begin{equation}\label{eq:rotation-separation}
  \max_{1\leq a\leq m}|\ee{r\alpha_a}-1|\geq\delta
  \qquad(r\in R).
\end{equation}
Then there exist a standard nonatomic probability space
$(X,\cB,\mu)$, an invertible weakly mixing probability-preserving
transformation $T:X\to X$, and a Borel set $A\subset X$ such that
\[
  \mu(A)>0,
  \qquad A\cap T^{-r}A=\varnothing\quad(r\in R).
\]
The system can be chosen to be a finite product of two-sided 
complex Gaussian shifts.
\end{theorem}

Condition~\eqref{eq:tail-gcd} says that arbitrarily large integers
divide all but finitely many members of each $E_i$. 
Condition~\eqref{eq:rotation-separation} says that $R$ is not Bohr
recurrent; compare the definition in~\cite[Section~1.1]{FHK}.
The divisibility condition alone does not suffice, as the example in
Remark~\ref{rem:recurrence-boundary} shows. For $S_{\cP}$, assign each integer to a prime with maximal exponent.
The resulting sets satisfy~\eqref{eq:tail-gcd}, and a prime outside
$\cP$ gives~\eqref{eq:rotation-separation}.
Section~\ref{sec:applications} also gives an infinitely generated
example and proves non-recurrence for $(S_{\cP}+h)\cap\N$
for every fixed $h\in\Z$.

The Gaussian construction associates continuous measures on the circle
with weakly mixing systems; see \cite[Section~2]{Grivaux}.
Related constructions of perfect sets and continuous measures
under divisibility assumptions appear in Eisner and
Grivaux~\cite[Propositions~3.8 and~3.9]{EG}.
Grivaux's non-recurrence criterion
\cite[Theorems~1.1 and~2.6]{Grivaux} combines separation by a rotation
with uniform approximation: for an infinite set $Q\subset\N$
containing $1$, one requires unit complex numbers $\lambda_k\ne1$
such that $\sup_{r\in Q}|\lambda_k^r-1|\to0$. For a single prime, weakly mixing non-recurrence, including the corresponding result for integer translates, follows from Grivaux~\cite[Theorem~1.2]{Grivaux}.
For multiplicative semigroups, Badea and Grivaux
\cite[Section~6.1]{BG} group the elements according to a maximal
prime exponent and construct continuous measures whose Fourier
coefficients tend to $1$ within each group.
This convergence alone does not ensure that
\[
\sum_{r\in Q}\int_\T|\ee{rt}-1|^2\,d\nu(t)
\]
is small enough for the Gaussian argument.
It also does not give the uniform approximation
required by Grivaux's criterion.

Our proof starts with a finite partition of $S_{\cP}$.
Following Badea and Grivaux~\cite[Section~6.1]{BG}, we assign
each element to a prime with maximal exponent in its
factorization. In each part,
the exponent of the chosen prime tends to infinity as the
elements tend to infinity. Thus every power of that prime
divides all sufficiently large elements of the part.

Fix $\eta>0$, to be chosen sufficiently small below,
and one part $E$. Choose denominators $D_j$ and integers
$ 1=N_0<N_1<N_2<\cdots$
in the order
$D_1,N_1,D_2,N_2,\ldots,$
so that $D_j$ divides every $r\in E$ with $r\ge N_j$.
For $\epsilon=0,1$, let $Q_\epsilon$ consist of those
$r\in E$ lying in an interval $[N_j,N_{j+1})$
with $j\equiv\epsilon\pmod2$.
For each $Q_\epsilon$, consider the random series
\[
\Theta_\epsilon
=
\sum_{\substack{\ell\ge1\\ \ell\equiv\epsilon\pmod2}}
\frac{\omega_\ell}{D_\ell},
\qquad \epsilon=0,1,
\]
where the digits $\omega_\ell$ are independent, each equal
to $0$ or $1$ with probability $1/2$. We take the sums
modulo one. Choose the denominators large enough and make them grow
fast enough that, for each $\epsilon$, different choices
of the digits used in $\Theta_\epsilon$ give different
points of $\T$. Each such digit sequence has probability
zero, so the distribution of each series has no atoms.

Now take $r\in Q_\epsilon\cap[N_j,N_{j+1})$,
so that $j\equiv\epsilon\pmod2$.
The denominators of the terms of $\Theta_\epsilon$
with indices at most $j$ divide $r$. After multiplication by $r$, these terms are integers.
They therefore do not change the value of
$\ee{r\Theta_\epsilon}$, which depends only on the
remaining terms.
Since $j+1$ has the opposite parity, the remaining
terms have indices $j+2,j+4,\ldots$.

When we choose $D_{j+2}$, the endpoint $N_{j+1}$ is
already fixed, and every $r$ in the interval satisfies
$r<N_{j+1}$. We can therefore choose $D_{j+2}$ and the
later denominators large enough to make
$|\ee{r\Theta_\epsilon}-1|$ as small as needed,
uniformly for all $r\in Q_\epsilon\cap[N_j,N_{j+1})$
and all choices of the digits. Since there are finitely many
such $r$, we can also make the sum of the squared bounds
over the interval as small as needed. We choose the
bounds for these interval sums so that their total is
less than $\eta$. Since the bounds do not depend on the
digits, the same estimates hold on the supports of
the measures.
Grivaux~\cite[Fact~4.2 and Section~4]{Grivaux} also uses
alternating finite intervals to obtain summable errors.
Here we choose each denominator to divide all sufficiently
large elements of the corresponding part.

Thus, for every $\eta>0$, we obtain a finite partition
of $S_{\cP}$ and, for each part $Q$, a continuous probability
measure $\nu$ on $\T$ such that
\[
\sum_{r\in Q}
\sup_{t\in\supp\nu}|\ee{rt}-1|^2<\eta.
\]
The partition and measures may depend on $\eta$.

We now use this estimate to build a weakly mixing system
with no returns along $Q$. Take a prime $q\notin\cP$ and
put $\zeta=\ee{1/q}$. No element of $S_{\cP}$ is divisible
by $q$. Hence there is a constant $\delta>0$ such that
$|\zeta^r-1|\ge\delta
$ for all $r\in S_{\cP}.$
Translate $\nu$ by $1/q$ to obtain a continuous probability
measure $\sigma$. Let $(X,\mu,T)$ be the complex Gaussian
shift associated with $\sigma$, with coordinates
$Z_k=Z_0\circ T^k$. The shift is invertible and weakly mixing.
Its coordinates satisfy
\[
\E|Z_r-\zeta^rZ_0|^2
=
\int_\T|\ee{rt}-1|^2\,d\nu(t).
\]
Thus the sum of these mean square errors over $r\in Q$
is less than $\eta$.

Set $a=\delta/8$ and let
$C=\{|Z_0-1|<a\}.$
Since $Z_0$ is standard complex Gaussian, $\mu(C)>0$
depends only on $a$, not on $\nu$.
On $C\cap T^{-r}C$, we have$|Z_0-1|<a$,$|Z_r-1|<a$.
Together with $|\zeta^r-1|\ge\delta$, this gives $|Z_r-\zeta^rZ_0|
>
\delta-2a
>
\frac{\delta}{2}.$
By Markov's inequality,
\[
\sum_{r\in Q}\mu(C\cap T^{-r}C)
\le
\frac{4}{\delta^2}
\sum_{r\in Q}\E|Z_r-\zeta^rZ_0|^2
<
\frac{4\eta}{\delta^2}.
\]
Since $\mu(C)$ depends only on $a$, we apply the preceding
construction with $\eta$ small enough that
$4\eta/\delta^2<\mu(C)/2$.

Remove from $C$ all points that return to $C$ at a time
in $Q$, and set
$A=C\setminus\bigcup_{r\in Q}T^{-r}C.$
The sum bound gives $\mu(A)>\mu(C)/2>0$.
Also, $A\subset C$ and $A\cap T^{-r}C=\varnothing$
for every $r\in Q$. Therefore
$A\cap T^{-r}A=\varnothing$ for all $r\in Q$.
The bound on the sum is needed because we remove the
return events for all $r\in Q$.

Apply this construction to each part of the partition.
Take the product of the systems and the product of the
sets. The finite product is weakly mixing, and the product
set has positive measure. Every $r\in S_{\cP}$ belongs
to one part. The set in that coordinate has no return
at time $r$, so neither does the product set.

The arithmetic construction only uses arbitrarily large
integers that divide all sufficiently large elements of
each part. It therefore applies beyond finitely generated
multiplicative semigroups. If finitely many rotations give a uniform separation
from $1$, we further partition each part so that
one rotation gives this separation on each new part. The same Gaussian
construction and finite product argument then prove
Theorem~\ref{thm:general}.
The choice of denominators also gives the uniform
approximation used in Corollary~\ref{cor:translates}
to prove fixed integer translates via Grivaux's criterion.

\smallskip

\noindent\textbf{Organization of the paper.}
Section~\ref{sec:preliminaries} recalls the necessary preliminaries and establishes
a Gaussian criterion for non-recurrence.
Section~\ref{sec:arithmetic} constructs continuous measures with summable
Fourier errors and proves the general theorem.
Section~\ref{sec:applications} applies this result to multiplicative semigroups,
integer translates, and symbolic systems.
\section{Preliminaries and Gaussian non-recurrence}\label{sec:preliminaries}

\subsection{Probability measures and weak mixing}
A probability space $(X,\cB,\mu)$ consists of a set $X$, a
$\sigma$-algebra $\cB$, and a countably additive measure with $\mu(X)=1$.
We work with standard Borel spaces: these are measurable spaces
isomorphic to a complete separable metric space with its Borel
$\sigma$-algebra. A measure has no atoms if every measurable set of positive
measure contains a measurable subset of strictly smaller positive
measure. On a standard Borel space this is equivalent to every singleton
having measure zero. A Borel measure on $\T$ is called continuous when
it has no atoms.

For a measurable map $f:X\to Y$, its pushforward is
$(f_*\mu)(B)=\mu(f^{-1}(B)).$
For every nonnegative measurable function $h$, and every integrable $h$,
\[
 \int_Y h\,d(f_*\mu)=\int_X h\circ f\,d\mu.
\]
For a Borel probability measure on a separable metric space, its support is
\[
 \supp\mu=\{x:\mu(V)>0\text{ for every open neighbourhood }V\text{ of }x\}.
\]
It is closed and has full measure: its complement is a countable union
of measure-zero basic open sets. The measure has full support on $X$
if every nonempty open subset of $X$ has positive measure.
We write $1_A$ for the indicator of $A$.

An invertible measure-preserving system $(X,\cB,\mu,T)$ consists of a probability space
and a
measurable bijection $T:X\to X$ with measurable inverse and $T_*\mu=\mu$.
For $r\in\Z$, we use $T^{-r}A=\{x:T^rx\in A\}$.
The transformation is ergodic if
$\mu(T^{-1}A\mathbin{\triangle}A)=0$ implies $\mu(A)\in\{0,1\}$,
where $\triangle$ denotes symmetric difference. It is weakly mixing if, for all
bounded measurable complex functions $f,g$,
\begin{equation}\label{eq:wm}
 \frac1N\sum_{n=1}^N
 \left|\int f(g\circ T^n)\,d\mu
          -\int f\,d\mu\int g\,d\mu\right|\longrightarrow0.
\end{equation}
Equivalently, $T\times T$ is ergodic for $\mu\times\mu$;
see \cite{CFS}. Taking indicators gives the equivalent set formulation
with $\mu(A\cap T^{-n}B)-\mu(A)\mu(B)$.

For a Borel probability measure $\sigma$ on $\T$, write
\[
 \widehat\sigma(n)=\int_\T\ee{nt}\,d\sigma(t)\qquad(n\in\Z).
\]
Then $\widehat\sigma(0)=1$ and
$\widehat\sigma(-n)=\overline{\widehat\sigma(n)}$.
For a topological space $X$, we write $\mathcal B(X)$
for its Borel $\sigma$-algebra.
All sequence spaces below carry the product topology,
with $\C$ carrying its usual topology and $\{0,1\}$
the discrete topology.

\subsection{Random variables and Gaussian shifts}\label{sec:gaussian}
Let $(\Omega,\mathcal F,\rho)$ be a probability space.
A complex random variable is a measurable map $W:\Omega\to\C$.
For integrable $W$, write $\E_\rho W=\int_\Omega W\,d\rho$,
omitting the subscript when the measure is clear. The distribution of a random vector
$W:\Omega\to\C^m$ is the pushforward measure $W_*\rho$.

A standard complex Gaussian random variable is a random
variable of the form
\[
Z=\frac{X+iY}{\sqrt{2}},
\]
where $X$ and $Y$ are independent real random variables
with density $(2\pi)^{-1/2}e^{-t^2/2}$ on $\R$.
The density of $Z$ with respect to Lebesgue area measure
on $\C$ is $\pi^{-1}e^{-|z|^2}$. In particular,
\[
\E Z=\E Z^2=0,\qquad \E|Z|^2=1.
\]

\smallskip

The following Gaussian construction is classical;
see \cite[Appendix~C]{Kechris}.
We use its complex-valued formulation from
\cite[Theorem~2.3 and Proposition~2.5]{Ackelsberg}.

\begin{lemma}\label{lem:gaussian}
	Let $\sigma$ be a Borel probability measure on $\T$.
	There exists a Borel probability measure $\gamma_\sigma$
	on $\C^\Z$ such that the coordinate functions
	$Z_k(z)=z_k$, $k\in\Z$, are standard complex Gaussian
	random variables satisfying
	\[
	\begin{aligned}
		\E(Z_k\overline{Z_l})&=\widehat\sigma(k-l),\\
		\E(Z_kZ_l)&=0
	\end{aligned}
	\qquad(k,l\in\Z).
	\]
	The measure $\gamma_\sigma$ has no atoms.
	The shift
	\[
	T:\C^\Z\to\C^\Z,\qquad
	(Tz)_k=z_{k+1},
	\]
	is invertible and preserves $\gamma_\sigma$.
	If $\sigma$ has no atoms, then
	$(\C^\Z,\gamma_\sigma,T)$ is weakly mixing.
\end{lemma}

\begin{proof}
	Set
$\varphi(n)=\widehat\sigma(n),$ for any $ n\in\Z.$
	Since $\sigma$ is a probability measure,
	$\varphi(0)=1$. Moreover, for any
	$n_1,\ldots,n_m\in\Z$ and $c_1,\ldots,c_m\in\C$,
	\[
	\begin{aligned}
		\sum_{a,b=1}^m
		\overline{c_a}c_b\varphi(n_b-n_a)
	=
		\int_\T\left|
		\sum_{b=1}^m c_b\ee{n_bt}
		\right|^2\,d\sigma(t)\ge0.
	\end{aligned}
	\]
	Thus $\varphi$ is positive definite, and
	\cite[Theorem~2.3]{Ackelsberg} applies.

	That theorem gives a Borel probability measure
	$\gamma_\sigma$ on $\C^\Z$, invariant under the shift
$(Sz)_k=z_{k-1}.$
	The coordinate functions $Z_k(z)=z_k$ satisfy
	\[
	\begin{aligned}
		\E Z_k&=0,\\
		\E(Z_k\overline{Z_l})
		&=\varphi(k-l)=\widehat\sigma(k-l),\\
		\E(Z_kZ_l)&=0
	\end{aligned}
	\qquad(k,l\in\Z).
	\]
	Since these variables are Gaussian and
	$\widehat\sigma(0)=1$, each $Z_k$ is
	standard complex Gaussian.
	
	\smallskip
	
	Define $T=S^{-1}$. Then
$(Tz)_k=z_{k+1}.$
	Since $S$ is invertible and preserves
	$\gamma_\sigma$, so does $T$.
Since $Z_0$ is standard complex Gaussian, it has
density $\pi^{-1}e^{-|w|^2}$ on $\C$.
Thus $\gamma_\sigma(\{Z_0=a\})=0$ for every $a\in\C$.
For every $z=(z_k)_{k\in\Z}\in\C^\Z$, we have
$\{z\}\subseteq
\{w\in\mathbb C^{\mathbb Z}:Z_0(w)=z_0\}$
and hence
$\gamma_\sigma(\{z\})
\le\gamma_\sigma(\{Z_0=z_0\})=0.$
Therefore $\gamma_\sigma$ has no atoms.

\smallskip
	
	Finally, suppose that $\sigma$ has no atoms.
	By \cite[Proposition~2.5]{Ackelsberg},
	the Gaussian system associated with $\sigma$,
	namely $(\C^\Z,\gamma_\sigma,S)$,
	is weakly mixing.
	Since weak mixing is preserved under inversion,
	$(\C^\Z,\gamma_\sigma,T)$ is also weakly mixing.
\end{proof}

We call $(\C^\Z,\gamma_\sigma,T)$ the complex
Gaussian shift associated with $\sigma$.
For later use, note that $Z_0\circ T^n=Z_n$, so
\[
\int (Z_0\circ T^n)\overline{Z_0}\,
d\gamma_\sigma
=\widehat\sigma(n)
\qquad(n\in\Z).
\]

\subsection{A Gaussian non-recurrence criterion}

The following estimate uses the Gaussian non-recurrence
argument of Grivaux~\cite[Proposition~4.5]{Grivaux};
see also \cite[Section~4]{BDJLR}.
It combines separation by a rotation with a summable bound
on return events.
We formulate the argument for coordinate Gaussian shifts
and include explicit constants for later use.

\begin{proposition}\label{prop:summable}
	Let $Q\subset\N$ be nonempty, and suppose that
	$\zeta\in\C$ and $\delta>0$ satisfy
	\[
	|\zeta|=1,\qquad
	|\zeta^r-1|\ge\delta
	\quad(r\in Q).
	\]
	Set
	\[
	a=\frac{\delta}{8},
	\qquad
	c_\delta=a^2e^{-(1+a)^2}.
	\]
	Suppose there is a continuous Borel probability
	measure $\nu$ on $\T$ such that
	\[
	\sum_{r\in Q}\int_\T
	|\ee{rt}-1|^2\,d\nu(t)
	<\frac{\delta^2c_\delta}{8}.
	\]
	Then there exist an invertible weakly mixing
	complex Gaussian shift $(X,\cB,\mu,T)$ and
	a Borel set $A\subset X$ satisfying
	\[
	\mu(A)>\frac{c_\delta}{2},
	\qquad
	A\cap T^{-r}A=\varnothing
	\quad(r\in Q).
	\]
\end{proposition}

\begin{proof}
Choose $\alpha\in\T$ with $\zeta=\ee{\alpha}$.
Define a Borel probability measure $\sigma$ on
$\T$ by
\[
\sigma(B)=\nu(B-\alpha)
\]
for every Borel set $B\subset\T$.
Since $\nu$ is continuous, so is $\sigma$.
Moreover,
\[
\begin{aligned}
	\widehat\sigma(r)
	=\int_\T \ee{r(t+\alpha)}\,d\nu(t)
	=\zeta^r\widehat\nu(r).
\end{aligned}
\]
By Lemma~\ref{lem:gaussian}, let
\[
X=\C^\Z,\qquad
\cB=\mathcal B(\C^\Z),\qquad
\mu=\gamma_\sigma,
\]
and define
\[
T:X\to X,\qquad
(Tz)_k=z_{k+1}.
\]
Then $(X,\cB,\mu,T)$ is an invertible weakly
mixing measure preserving system. For $k\in\Z$,
let $Z_k:X\to\C$ be the coordinate function
$Z_k(z)=z_k$. Each $Z_k$ is a standard complex Gaussian
random variable. In particular,
\[
Z_0(T^r z)=Z_r(z),
\qquad
\E_\mu(Z_r\overline{Z_0})
=\widehat\sigma(r).
\]

Since $\E_\mu|Z_r|^2=\E_\mu|Z_0|^2=1$,
we have
\[
\begin{aligned}
	\E_\mu|Z_r-\zeta^rZ_0|^2
	&=2-2\operatorname{Re}
	\bigl(\overline\zeta^{\,r}
	\E_\mu(Z_r\overline{Z_0})\bigr)\\
	&=2-2\operatorname{Re}
	\bigl(\overline\zeta^{\,r}\widehat\sigma(r)\bigr)\\
	&=2-2\operatorname{Re}\widehat\nu(r)\\
	&=\int_\T|\ee{rt}-1|^2\,d\nu(t).
\end{aligned}
\]
	
	Let
	$
	C=\{x\in X:|Z_0(x)-1|<a\}.
	$
	Since $Z_0$ is standard complex Gaussian, it has
	density $\pi^{-1}e^{-|z|^2}$ on $\C$. Thus,
we have
	\[
	\begin{aligned}
		\mu(C)
		&=\frac1\pi\int_{|z-1|<a}
		e^{-|z|^2}\,dA(z)\\
		&\ge \frac1\pi e^{-(1+a)^2}\cdot\pi a^2\\
		&=a^2e^{-(1+a)^2}=c_\delta,
	\end{aligned}
	\]
	where $dA$ denotes Lebesgue area measure on $\C$.
	
	Fix $r\in Q$. If $x\in C\cap T^{-r}C$,
	then
	\[
	|Z_0(x)-1|<a,\qquad |Z_r(x)-1|<a.
	\]
	Since $|\zeta^r-1|\ge\delta$, we have
	\[
	\begin{aligned}
		|Z_r(x)-\zeta^rZ_0(x)|
		&\ge|\zeta^r-1|
		-|Z_r(x)-1|-|Z_0(x)-1|\\
		&>\delta-2a
		=\frac{3\delta}{4}
		>\frac{\delta}{2}.
	\end{aligned}
	\]
	Thus Markov's inequality gives
	\[
	\begin{aligned}
		\mu(C\cap T^{-r}C)
		\le\frac4{\delta^2}
		\E|Z_r-\zeta^rZ_0|^2
		=\frac4{\delta^2}
		\int_\T|\ee{rt}-1|^2\,d\nu(t).
	\end{aligned}
	\]
	
	Now define
	\[
	A=C\setminus\bigcup_{r\in Q}T^{-r}C.
	\]
	By the assumption on $\nu$,
	\[
	\begin{aligned}
		\mu(A)
		&\ge\mu(C)-
		\sum_{r\in Q}\mu(C\cap T^{-r}C)\\
		&>c_\delta-
		\frac4{\delta^2}\cdot
		\frac{\delta^2c_\delta}{8}\\
		&=\frac{c_\delta}{2}.
	\end{aligned}
	\]
	Finally, $A\subset C$ and
	$A\cap T^{-r}C=\varnothing$ for every $r\in Q$.
	Since $T^{-r}A\subset T^{-r}C$, we obtain
	\[
	A\cap T^{-r}A=\varnothing
	\qquad(r\in Q).
	\]
\end{proof}

\subsection{Products and finite sets of times}
We use the following finite product argument to combine the sets of
times; compare \cite[Section~4]{BDJLR}.

\begin{lemma}\label{lem:product}
Let $(X_j,\cB_j,\mu_j,T_j)$, $1\leq j\leq l$, be finitely many
invertible weakly mixing probability-preserving systems, and let
$(X,\cB,\mu,T)$ be their product. Then $T$ is invertible, probability
preserving, and weakly mixing. If $Q_j\subset\N$ and $A_j\in\cB_j$
satisfy $\mu_j(A_j)>0$ and
$A_j\cap T_j^{-r}A_j=\varnothing$ for every $r\in Q_j$, then
$A=\prod_{j=1}^lA_j$ satisfies
\[
 \mu(A)>0,\qquad A\cap T^{-r}A=\varnothing
 \quad\left(r\in\bigcup_{j=1}^lQ_j\right).
\]
\end{lemma}

\begin{proof}
	The product of finitely many weakly mixing systems
	is weakly mixing; see
	\cite[Section~4]{BDJLR}.
	Moreover,
	\[
	\mu(A)=\prod_{j=1}^l\mu_j(A_j)>0,
	\qquad
	A\cap T^{-r}A
	=\prod_{j=1}^l(A_j\cap T_j^{-r}A_j).
	\]
	For $r\in\bigcup_{j=1}^l Q_j$, there exists
	$j$ such that
	\[
	A_j\cap T_j^{-r}A_j=\varnothing.
	\]
	Hence $A\cap T^{-r}A=\varnothing$.
\end{proof}

\begin{lemma}\label{lem:finite-times}
For every finite set $H\subset\N$, there are an invertible weakly
mixing two-sided complex Gaussian shift $(X,\cB,\mu,T)$ and a Borel
set $B$ of positive measure such that
$B\cap T^{-r}B=\varnothing$ for all $r\in H$.
\end{lemma}

\begin{proof}
	Take the Gaussian shift in Lemma~\ref{lem:gaussian},
	with $\sigma$ equal to normalized Lebesgue measure on $\T$.
	Its coordinates are independent standard complex Gaussian
	variables, and the shift is weakly mixing.
	Put $L=\max(H\cup\{0\})$ and
	\[
	B=\{|Z_0-1|<1/4\}
	\cap\bigcap_{k=1}^L\{|Z_k|<1/4\}.
	\]
By independence and the positivity of the Gaussian
density, $\mu(B)>0$.
For every $r\in H$, if $x\in B\cap T^{-r}B$,
then
\[
|Z_r(x)|<1/4,\qquad
|Z_r(x)-1|
=|Z_0(T^r x)-1|<1/4.
\]
Hence
\[
1\le |Z_r(x)|+|Z_r(x)-1|<1/2,
\]
a contradiction. Thus
$B\cap T^{-r}B=\varnothing$ for all $r\in H$.
\end{proof}

\section{Divisibility and spectral measures}
\label{sec:arithmetic}
We divide the given set of positive integers into two parts
and construct a continuous probability measure for each part.
For each measure, we bound $|\ee{rt}-1|$ uniformly over
its support and make the sum of the squared bounds
over all integers in the corresponding part as small as needed.
We will also use this construction to prove integer
translates in Corollary~\ref{cor:translates}.

The reciprocal-denominator series below are closely related
to the constructions of Eisner and
Grivaux~\cite[Propositions~3.8 and~3.9]{EG}
and Maucourant~\cite[Lemma~4.2]{Maucourant}.
Maucourant uses separate even- and odd-indexed series:
divisibility makes the earlier terms integers after multiplication,
while the remaining terms are small.
Grivaux~\cite[Fact~4.2]{Grivaux} also uses alternating
finite intervals to obtain summable estimates.
Here we use random digits and choose denominators dividing
all sufficiently large elements of the relevant set,
so that the resulting estimates are uniform on the
supports of the measures.

\begin{lemma}\label{lem:measures}
Let $E\subset\N$ be infinite and suppose that
\[
 \gcd\{r\in E:r\geq N\}\longrightarrow\infty\quad(N\to\infty).
\]
Then for every $\eta>0$, there are a partition $E=C_0\sqcup C_1$ into
infinite sets and continuous Borel probability measures $\nu_0,\nu_1$
on $\T$ such that, with $Q_\epsilon=C_\epsilon\cup\{1\}$,
\[
 \sum_{r\in Q_\epsilon}
       \sup_{t\in\supp\nu_\epsilon}|\ee{rt}-1|^2<\eta
 \qquad(\epsilon=0,1).
\]
For each of these fixed sets $Q_\epsilon$, there are unit complex
numbers $\lambda_{\epsilon,k}\ne1$ satisfying
\begin{equation}\label{eq:uniform-approximation}
 \lambda_{\epsilon,k}\longrightarrow1,\qquad
 \sup_{r\in Q_\epsilon}|\lambda_{\epsilon,k}^{\,r}-1|
     \longrightarrow0\quad(k\to\infty).
\end{equation}
The partition and measures may depend on $\eta$.
\end{lemma}

\begin{proof}
For each $N$, the integer
$\gcd\{r\in E:r\ge N\}$ divides all but finitely many
members of $E$. Hence the hypothesis implies that
\[
\mathcal D
=
\{D\in\N:
D\text{ divides all but finitely many members of }E\}
\]
is unbounded. Set $D_0=N_0=1$. After $D_{j-1}$ and $N_{j-1}$
have been fixed, choose $D_j\in\mathcal D$ so that
\[
 D_j\geq4D_{j-1},\qquad D_j\geq2^jN_{j-1},\qquad
 \frac{16\pi^2N_{j-1}^3}{D_j^2}<\eta\,2^{-j}
 \quad(j\geq1),
\]
and then choose an integer $N_j>N_{j-1}$ such that
\[
 D_j\mid r\quad(r\in E,\ r\geq N_j),\qquad
 E\cap[N_{j-1},N_j)\ne\varnothing.
\]
Such an $N_j$ exists because $D_j\in\mathcal D$ and $E$ is unbounded.
No divisibility between successive $D_j$ is required. Define
\[
 C_\epsilon=E\cap\bigcup_{\substack{j\geq0\\j\equiv\epsilon\pmod2}}
                  [N_j,N_{j+1}),\qquad
 I_\epsilon=\{\ell\geq1:\ell\equiv\epsilon\pmod2\}.
\]
Every interval meets $E$, so $E=C_0\sqcup C_1$ and both sets are
infinite. 
For $\epsilon\in\{0,1\}$, let
$
\Omega_\epsilon=\{0,1\}^{I_\epsilon}
$
be equipped with the product probability measure
$\mathbb P_\epsilon$ under which each coordinate takes
the values $0$ and $1$ with equal probability.
For $\omega=(\omega_\ell)_{\ell\in I_\epsilon}
\in\Omega_\epsilon$, define
\[
\Theta_\epsilon(\omega)
=\sum_{\ell\in I_\epsilon}\frac{\omega_\ell}{D_\ell}.
\]
We define a probability measure $\nu_\epsilon$
on $\mathbb T$ by
\[
\nu_\epsilon(B)
=
\mathbb P_\epsilon\bigl(
\{\omega\in\Omega_\epsilon:
\Theta_\epsilon(\omega)\bmod1\in B\}
\bigr)
\]
for every Borel set $B\subset\mathbb T$.
Since $D_{\ell+1}\ge 4D_\ell$ and $D_1\ge4$,
we have
\[
0\le\Theta_\epsilon(\omega)
\le\sum_{\ell=1}^\infty\frac1{D_\ell}
\le\frac1{D_1}\sum_{\ell=0}^\infty4^{-\ell}
=\frac4{3D_1}\le\frac12.
\]
The series converges uniformly on the compact
space $\Omega_\epsilon$.
Thus $\Theta_\epsilon$ is continuous on $\Omega_\epsilon$. Consequently,
\[
K_\epsilon
=
\left\{
\sum_{\ell\in I_\epsilon}
\frac{\omega_\ell}{D_\ell}\pmod1:
\omega_\ell\in\{0,1\}
\right\}
\]
is compact. By the definition of $\nu_\epsilon$,
$
\nu_\epsilon(K_\epsilon)=1,
$ and $
\operatorname{supp}\nu_\epsilon\subset K_\epsilon.
$

Let $\omega\ne\omega'$, and let
$\ell\in I_\epsilon$ be the first index at which they
differ. Then
\[
\begin{aligned}
	|\Theta_\epsilon(\omega)-\Theta_\epsilon(\omega')|
	&=
	\left|
	\frac{\omega_\ell-\omega'_\ell}{D_\ell}
	+\sum_{\substack{k\in I_\epsilon\\k>\ell}}
	\frac{\omega_k-\omega'_k}{D_k}
	\right|
	\ge \frac1{D_\ell}
	-\sum_{\substack{k\in I_\epsilon\\k>\ell}}\frac1{D_k}\\
	&\ge \frac1{D_\ell}-\frac2{D_{\ell+2}}
	\ge \frac7{8D_\ell}>0.
\end{aligned}
\]
Hence, for $\omega\ne\omega'$,
$
0<
|\Theta_\epsilon(\omega)-\Theta_\epsilon(\omega')|
\le\frac12<1.
$
In particular,
$
\Theta_\epsilon(\omega)-\Theta_\epsilon(\omega')
\notin\mathbb Z.
$
Thus
$
\Theta_\epsilon(\omega)\not\equiv
\Theta_\epsilon(\omega')\pmod1,
$
and then for every $t\in K_\epsilon$, there is a unique
$\omega\in\Omega_\epsilon$ such that
\[
\Theta_\epsilon(\omega)\equiv t\pmod1.
\]

Finally, we show that $\nu_\epsilon$ has no atoms.
Fix $a=(a_\ell)_{\ell\in I_\epsilon}\in\Omega_\epsilon$.
For every integer $M\ge1$, independence gives
\[
\begin{aligned}
	\mathbb P_\epsilon(\{\omega=a\})
	&\le
	\mathbb P_\epsilon\bigl(
	\omega_{2i-\epsilon}=a_{2i-\epsilon}
	\text{ for }1\le i\le M
	\bigr)\\
	&=\prod_{i=1}^{M}
	\mathbb P_\epsilon\bigl(
	\omega_{2i-\epsilon}=a_{2i-\epsilon}
	\bigr)
	=2^{-M}.
\end{aligned}
\]
Letting $M\to\infty$, we obtain
$
\mathbb P_\epsilon(\{\omega=a\})=0.
$
Since $\omega\mapsto\Theta_\epsilon(\omega)\pmod1$
is injective, for every $t\in\mathbb T$ the set
\[
\{\omega\in\Omega_\epsilon:
\Theta_\epsilon(\omega)\equiv t\pmod1\}
\]
contains at most one element. Since every singleton
in $\Omega_\epsilon$ has probability zero, we obtain
\[
\nu_\epsilon(\{t\})
=\mathbb P_\epsilon\bigl(
\Theta_\epsilon(\omega)\equiv t\pmod1
\bigr)=0.
\]
Hence $\nu_\epsilon$ is a continuous probability measure.

Fix $r\in E\cap[N_j,N_{j+1})$, where
$j\equiv\epsilon\pmod2$. For every
$\ell\in I_\epsilon$ with $\ell\le j$, we have
$r\ge N_j\ge N_\ell$. By the choice of $N_\ell$,
$
D_\ell\mid r,
$ $
\frac{r\omega_\ell}{D_\ell}\in\mathbb Z.
$
Since $j+1\notin I_\epsilon$, it follows that
\[
\begin{aligned}
	r\Theta_\epsilon(\omega)
	&=\sum_{\substack{\ell\in I_\epsilon\\\ell\le j}}
	\frac{r\omega_\ell}{D_\ell}
	+r\sum_{\substack{\ell\in I_\epsilon\\\ell\ge j+2}}
	\frac{\omega_\ell}{D_\ell}\\
	&\equiv
	r\sum_{\substack{\ell\in I_\epsilon\\\ell\ge j+2}}
	\frac{\omega_\ell}{D_\ell}
	\pmod1.
\end{aligned}
\]
When $j=0$, we have $\epsilon=0$ and
$I_\epsilon=\{2,4,6,\ldots\}$, so the first sum
is empty and the same identity holds.

Using $|e^{2\pi iu}-1|\le2\pi|u|$ and
$0\le\omega_\ell\le1$, we obtain
\[
\begin{aligned}
	|\ee{r\Theta_\epsilon(\omega)}-1|
	&\le
	2\pi r
	\sum_{\substack{\ell\in I_\epsilon\\\ell\ge j+2}}
	\frac{\omega_\ell}{D_\ell}
	\le
	2\pi r
	\sum_{\ell\ge j+2}\frac1{D_\ell}\\
	&\le
	\frac{2\pi r}{D_{j+2}}
	\sum_{k=0}^\infty4^{-k}
	\le
	\frac{4\pi N_{j+1}}{D_{j+2}},
\end{aligned}
\]
where the last inequality uses $r<N_{j+1}$.
Since this estimate holds for every
$\omega\in\Omega_\epsilon$, the definition of
$K_\epsilon$ gives
\[
\sup_{t\in K_\epsilon}
|\ee{rt}-1|
\le\frac{4\pi N_{j+1}}{D_{j+2}}.
\] Since the interval contains
at most $N_{j+1}$ integers,
\[
 \sum_{r\in E\cap[N_j,N_{j+1})}
       \sup_{t\in K_\epsilon}|\ee{rt}-1|^2
 \leq\frac{16\pi^2N_{j+1}^3}{D_{j+2}^2}
 <\eta\,2^{-j-2}.
\]
Also,
\[
 \sup_{t\in K_\epsilon}|\ee{t}-1|^2
 \leq\frac{16\pi^2}{D_1^2}<\frac\eta2.
\]
Combining the estimates for $r\in C_\epsilon$ and $r=1$,
we obtain
\[
\begin{aligned}
	\sum_{r\in C_\epsilon\cup\{1\}}
	\sup_{t\in K_\epsilon}|\ee{rt}-1|^2
	&\le
	\sum_{r\in C_\epsilon}
	\sup_{t\in K_\epsilon}|\ee{rt}-1|^2
	+\sup_{t\in K_\epsilon}|\ee{t}-1|^2\\
	&<\frac{\eta}{2}+\frac{\eta}{2}
	=\eta.
\end{aligned}
\]

Finally, for each $J\ge2$ with $J\equiv\epsilon\pmod2$,
define
$\lambda_J=\ee{1/D_J}\ne1$.
For $r\in C_\epsilon\cap[N_j,N_{j+1})$,
\[
 \begin{cases}
 D_J\mid r,&j\geq J,\\
 r<N_{J-1},&j<J,
 \end{cases}
\]
because $j<J$ implies $j\leq J-2$. Hence, including $r=1$,
\[
 \sup_{r\in Q_\epsilon}|\lambda_J^r-1|
 \leq\frac{2\pi N_{J-1}}{D_J}\leq2\pi2^{-J}.
\]
Since $D_J\to\infty$, we have $\lambda_J\to1$. Taking $J=2k+\epsilon$ and
\[
\lambda_{\epsilon,k}
=
\ee{1/D_{2k+\epsilon}},
\qquad k\ge1,
\]
proves \eqref{eq:uniform-approximation}.
\end{proof}

Now, we prove our main result.

\begin{proof}[Proof of Theorem~\ref{thm:general}]
Let $\alpha_1,\ldots,\alpha_m\in\mathbb T$ and
$\delta>0$ satisfy \eqref{eq:rotation-separation},
and set
\[
c_\delta=(\delta/8)^2\exp\bigl(-(1+\delta/8)^2\bigr),
\qquad
\eta=\delta^2c_\delta/16.
\]
Apply Lemma~\ref{lem:measures} to each $E_i$, obtaining
$E_i=C_{i,0}\sqcup C_{i,1}$ and measures $\nu_{i,\epsilon}$.
For $r\in R$, put
\[
 a(r)=\min\{a:|\ee{r\alpha_a}-1|\geq\delta\},\qquad
 R_{i,\epsilon,a}=\{r\in C_{i,\epsilon}:a(r)=a\}.
\]
On each nonempty $R_{i,\epsilon,a}$, the number
$\zeta_a=\ee{\alpha_a}$ and the measure $\nu_{i,\epsilon}$ satisfy
\[
 \inf_{r\in R_{i,\epsilon,a}}|\zeta_a^r-1|\geq\delta,\qquad
 \sum_{r\in R_{i,\epsilon,a}}
       \int_\T|\ee{rt}-1|^2\,d\nu_{i,\epsilon}(t)
 <\eta<\frac{\delta^2c_\delta}{8}.
\]
Then
\[
R=F\cup\bigcup_{i=1}^d
\bigcup_{\epsilon=0}^1
\bigcup_{a=1}^m R_{i,\epsilon,a}.
\]
Write the nonempty sets among $F$ and
$R_{i,\epsilon,a}$ as $Q_1,\ldots,Q_L$.
Then $L\le 2dm+1$.
By Proposition~\ref{prop:summable} and
Lemma~\ref{lem:finite-times}, for each $Q_j$
there is a weakly mixing Gaussian shift
$(X_j,\mu_j,T_j)$ and a Borel set $A_j$ such that
\[
\mu_j(A_j)>0,\qquad
A_j\cap T_j^{-r}A_j=\varnothing
\quad(r\in Q_j).
\]
Now set
\[
(X,\mu,T)=\prod_{j=1}^L(X_j,\mu_j,T_j),
\qquad A=\prod_{j=1}^L A_j.
\]
By Lemma~\ref{lem:product}, $T$ is weakly mixing,
and
\[
\mu(A)=\prod_{j=1}^L\mu_j(A_j)>0,
\qquad
A\cap T^{-r}A=\varnothing
\quad(r\in R).
\]
This completes the proof.
\end{proof}

\begin{remark}\label{rem:recurrence-boundary}
	Condition~\eqref{eq:tail-gcd} alone does not
	guarantee non-recurrence. Consider
	\[
	R_*=\bigcup_{k\ge1}
	\{j2^k:1\le j\le k\}.
	\]
	We first show that $R_*$ is a set of recurrence.
	Let $(X,\cB,\mu,T)$ be an invertible
	probability-preserving system, and let
	$A\in\cB$ satisfy $\mu(A)>0$.
	Choose $k\ge1$ such that
	$(k+1)\mu(A)>1$.
	Since
	\[
	\sum_{j=0}^k\mu(T^{-j2^k}A)
	=(k+1)\mu(A)>1,
	\]
	there exist $0\le i<j\le k$ such that
	\[
	\mu(T^{-i2^k}A\cap T^{-j2^k}A)>0.
	\]
	By invariance of $\mu$,
	\[
	\mu(A\cap T^{-(j-i)2^k}A)>0.
	\]
	Since $(j-i)2^k\in R_*$, the set $R_*$
	is recurrent; compare
	\cite[Remark~3.1]{Grivaux}.
	
	On the other hand, for every $K\ge1$,
	all elements of $R_*$ outside the finitely
	many sets
	\[
	\{j2^k:1\le j\le k\},
	\qquad 1\le k<K,
	\]
	are divisible by $2^K$. Hence
	\[
	\lim_{N\to\infty}
	\gcd\{r\in R_*:r\ge N\}=\infty.
	\]
	Thus $R_*$ satisfies the divisibility condition
	\eqref{eq:tail-gcd} but cannot satisfy
	\eqref{eq:rotation-separation}.
\end{remark}

\section{Applications}\label{sec:applications}

\subsection{Finitely generated semigroups}
We use the partition according to a maximal exponent from
\cite[Section~6.1]{BG}.

\begin{lemma}\label{lem:valuation-cover}
Let $R\subset\N$ be infinite, let $p_1,\ldots,p_d$ be distinct
primes, with $d\geq1$, and write
$v_{p_i}(r)=\max\{k\in\Nzero:p_i^k\mid r\}$.
Suppose that
\[
 \#\{r\in R:\max_{1\leq i\leq d}v_{p_i}(r)<K\}<\infty
 \quad\text{for every integer }K\geq1.
\]
Then $R=F\cup E_1\cup\cdots\cup E_s$, where $F$ is finite,
$1\leq s\leq d$, each $E_i$ is infinite, and
\[
 \gcd\{r\in E_i:r\geq N\}\longrightarrow\infty
 \quad(N\to\infty,\ 1\leq i\leq s).
\]
\end{lemma}

\begin{proof}
	For each $r\in R$, choose the smallest index
	$i$ for which $v_{p_i}(r)$ is maximal. More precisely, set
	$
	i(r)=\min\left\{
	i:v_{p_i}(r)=\max_{1\le j\le d}v_{p_j}(r)
	\right\},
$ $
	G_i=\{r\in R:i(r)=i\}.
$
	Then
	$$
	R=\bigsqcup_{i=1}^d G_i.
	$$
	
	Fix $1\le i\le d$. We show that if $G_i$ is infinite, then
	\[
	\lim_{N\to\infty}
	\gcd\{r\in G_i:r\ge N\}=\infty.
	\]
	
	Let $K\ge1$ be an integer. If $r\in G_i$ and
	$p_i^K\nmid r$, then
	\[
	\max_{1\le j\le d}v_{p_j}(r)
	=v_{p_i}(r)<K.
	\]
	Consequently,
	\[
	\{r\in G_i:p_i^K\nmid r\}
	\subset
	\left\{
	r\in R:
	\max_{1\le j\le d}v_{p_j}(r)<K
	\right\}.
	\]
	The set on the right is finite by assumption.
	Hence there exists $N_K\ge1$ such that for any $r\in G_i,\ r\ge N_K,$ we have
$p_i^K\mid r$.
	
	Now suppose that $G_i$ is infinite. For every
	$N\ge N_K$, the set
	$\{r\in G_i:r\ge N\}$ is nonempty, and all
	its elements are divisible by $p_i^K$.
	Therefore
	\[
	p_i^K\mid
	\gcd\{r\in G_i:r\ge N\},
	\]
	and in particular,
	\[
	\gcd\{r\in G_i:r\ge N\}\ge p_i^K
	\qquad(N\ge N_K).
	\]
	Since $K$ is arbitrary, we obtain
	\[
	\lim_{N\to\infty}
	\gcd\{r\in G_i:r\ge N\}=\infty.
	\]
	
	Finally, let $F$ be the union of the finite
	sets among $G_1,\ldots,G_d$, and denote the
	remaining sets by $E_1,\ldots,E_s$.
	Since $R$ is infinite, at least one $G_i$
	is infinite, so $1\le s\le d$. Thus
	\[
	R=F\cup E_1\cup\cdots\cup E_s,
	\]
	where $F$ is finite and each $E_i$ is infinite
	and satisfies
	\[
	\lim_{N\to\infty}
	\gcd\{r\in E_i:r\ge N\}=\infty.
	\]
\end{proof}

\begin{corollary}\label{cor:generators}
Let $l\geq1$, let $a_1,\ldots,a_l$ be positive integers, and put
\[
 S=\{a_1^{n_1}\cdots a_l^{n_l}:n_1,\ldots,n_l\in\Nzero\}.
\]
There exist an invertible weakly mixing probability-preserving system
$(X,\cB,\mu,T)$ and a Borel set $A\subset X$ such that
\[
 \mu(A)>0,\qquad A\cap T^{-r}A=\varnothing\quad(r\in S).
\]
\end{corollary}

\begin{proof}[Proof of Theorem~\ref{thm:main} and
	Corollary~\ref{cor:generators}]
	Let $\cP=\{p_1,\ldots,p_d\}$.
	Every $r\in S_{\cP}$ has a unique representation
	\[
	r=p_1^{n_1}\cdots p_d^{n_d},
	\qquad n_1,\ldots,n_d\in\Nzero,
	\]
	with $v_{p_i}(r)=n_i$ for $1\le i\le d$.
	For every integer $K\ge1$, we therefore have
	\[
	\begin{aligned}
		&\#\left\{r\in S_{\cP}:
		\max_{1\le i\le d}v_{p_i}(r)<K\right\}\\
		&\qquad=
		\#\left\{(n_1,\ldots,n_d):
		0\le n_i<K,\ 1\le i\le d\right\}
		=K^d<\infty.
	\end{aligned}
	\]
	By Lemma~\ref{lem:valuation-cover}, we can write
	\[
	S_{\cP}=F\cup E_1\cup\cdots\cup E_s,
	\]
	where $F$ is finite, $1\le s\le d$, and each
	$E_i$ is infinite and satisfies
	\[
	\lim_{N\to\infty}
	\gcd\{r\in E_i:r\ge N\}=\infty.
	\]
	Thus the first condition of
	Theorem~\ref{thm:general} holds.
	
	Next, choose a prime $q\notin\cP$.
	By the definition of $S_{\cP}$,
	we have $q\nmid r$ for every $r\in S_{\cP}$.
	Hence, for each $r\in S_{\cP}$, there exists
	an integer $b$ with $1\le b\le q-1$ such that
	$r\equiv b\pmod q$. Consequently,
	\[
	\begin{aligned}
		|\ee{r/q}-1|
		=|\ee{b/q}-1|
		=2\sin\left(\frac{\pi b}{q}\right)
		\ge 2\sin\left(\frac{\pi}{q}\right)>0.
	\end{aligned}
	\]
	Therefore \eqref{eq:rotation-separation}
	holds with
	\[
	m=1,\qquad
	\alpha_1=\frac1q,\qquad
	\delta=2\sin\left(\frac{\pi}{q}\right).
	\]
	Theorem~\ref{thm:general} now gives the
	required system and set, proving
	Theorem~\ref{thm:main}.
	
To prove Corollary~\ref{cor:generators},
choose a finite nonempty set of primes $\mathcal Q$
such that every $a_j$ can be written as
\[
a_j=\prod_{p\in\mathcal Q}p^{e_{j,p}},
\qquad e_{j,p}\in\Nzero.
\]
For any $n_1,\ldots,n_l\in\Nzero$, we have
\[
a_1^{n_1}\cdots a_l^{n_l}
=\prod_{p\in\mathcal Q}
p^{\sum_{j=1}^l n_je_{j,p}}
\in S_{\mathcal Q}.
\]
Thus $S\subset S_{\mathcal Q}$.
By Theorem~\ref{thm:main}, there exist an
invertible weakly mixing system
$(X,\cB,\mu,T)$ and a Borel set $A$
such that
\[
\mu(A)>0,\qquad
A\cap T^{-r}A=\varnothing
\quad(r\in S_{\mathcal Q}).
\]
Since $S\subset S_{\mathcal Q}$, the same
conclusion holds for every $r\in S$.
\end{proof}

\subsection{An infinitely generated semigroup}

\begin{corollary}\label{cor:infinite}
	Let $\ell_1<\ell_2<\cdots$ be the primes greater
	than $5$, and let $\Gamma$ be the multiplicative
	semigroup generated by
	\[
	2,\quad 3,\quad 2^j\ell_j\quad(j\ge1),
	\]
	including $1$.
	Then there exist an invertible weakly mixing
	probability-preserving system $(X,\cB,\mu,T)$
	and a Borel set $A\subset X$ such that
	\[
	\mu(A)>0,\qquad
	A\cap T^{-r}A=\varnothing
	\quad(r\in\Gamma).
	\]
	The system can be chosen as a finite product
	of two-sided  complex Gaussian shifts.
	Moreover,
	\[
	\{2^a3^b:a,b\in\Nzero\}\subset\Gamma,
	\qquad
	2^j\ell_j\in\Gamma\quad(j\ge1).
	\]
	In particular, $\Gamma$ is not finitely generated.
\end{corollary}

\begin{proof}
	Every $r\in\Gamma$ can be written as
	\[
	r=2^a3^b\prod_{j\ge1}(2^j\ell_j)^{c_j},
	\]
	where $a,b,c_j\in\Nzero$ and only finitely
	many $c_j$ are nonzero. Recall that $
	v_p(r)=\max\{k\in\mathbb N_0:p^k\mid r\}$. Since each $\ell_j$
	is a prime greater than $5$, we have
$v_2(r)=a+\sum_{j\ge1}jc_j,$ $v_3(r)=b.$
	
	Fix an integer $K\ge1$. Suppose that
$\max\{v_2(r),v_3(r)\}<K.$
	Then
$a+\sum_{j\ge1}jc_j<K,$ and $ b<K.$
	Since all the exponents are nonnegative,
	we obtain
$0\le a,b<K,c_j=0 $ for every $j\ge K,$
	and
$0\le c_j<K$ for every $1\le j<K.$
	Thus $r$ is determined by the $K+1$
	integers $a,b,c_1,\ldots,c_{K-1}$,
	each having at most $K$ possible values.
	Consequently,
	\[
	\#\{r\in\Gamma:
	\max\{v_2(r),v_3(r)\}<K\}
	\le K^{K+1}<\infty.
	\]
	Since $\Gamma$ is infinite, applying
	Lemma~\ref{lem:valuation-cover} with
	$p_1=2$ and $p_2=3$ gives
	\[
	\Gamma=F\cup E_1\cup\cdots\cup E_s,
	\]
	where $F$ is finite, $1\le s\le2$, and
	each $E_i$ is infinite and satisfies
	\[
	\lim_{N\to\infty}
	\gcd\{r\in E_i:r\ge N\}=\infty.
	\]
	
	Next, since $2,3,\ell_j$ are all different
	from $5$, we have $5\nmid r$ for every
	$r\in\Gamma$. Hence
	\[
	|\ee{r/5}-1|
	\ge 2\sin(\pi/5)>0
	\qquad \text{~~for any~}r\in\Gamma.
	\]
	Thus \eqref{eq:rotation-separation}
	holds with
	\[
	m=1,\qquad
	\alpha_1=\frac15,\qquad
	\delta=2\sin(\pi/5).
	\]
	Theorem~\ref{thm:general} now gives the
	required system and Borel set.
	
	Finally, since $2,3\in\Gamma$, we have
$\{2^a3^b:a,b\in\Nzero\}\subset\Gamma.$
	Moreover,
$\ell_j\mid 2^j\ell_j\in\Gamma$ for any $j\ge1.$
	If $\Gamma$ were generated by finitely many
	positive integers $b_1,\ldots,b_t$,
	every $\ell_j$ would divide
	$b_1\cdots b_t$, which is impossible
	since the primes $\ell_j$ are all distinct.
	Therefore $\Gamma$ is not finitely generated.
\end{proof}

\subsection{Integer translates}
We use the following consequence of
\cite[Theorems~1.1 and~2.6, and Section~3]{Grivaux}.
The proof of Grivaux's Theorem~1.1 uses the Hilbert-space
operator construction of Eisner and
Grivaux~\cite[Theorem~2.1]{EG}.

\begin{lemma}\label{lem:grivaux}
	Let $Q\subset\N$ be infinite with $1\in Q$.
	Suppose there exist unit complex numbers
	$\zeta$ and $\lambda_k\ne1$, $k\ge1$,
	such that
	\[
	\inf_{r\in Q}|\zeta^r-1|>0,
	\qquad
	\lim_{k\to\infty}
	\sup_{r\in Q}|\lambda_k^r-1|=0.
	\]
	Then there exist an invertible weakly mixing
	probability-preserving system $(X,\cB,\mu,T)$
	and a Borel set $A\subset X$ satisfying
	\[
	\mu(A)>0,\qquad
	A\cap T^{-r}A=\varnothing
	\quad(r\in Q).
	\]
\end{lemma}

\begin{corollary}\label{cor:translates}
Let $\cP=\{p_1,\ldots,p_d\}$ be a finite nonempty set of primes,
let $h\in\Z$, and put
\[
 S_{\cP}=\{p_1^{n_1}\cdots p_d^{n_d}:n_1,\ldots,n_d\in\Nzero\}.
\]
There exist an invertible weakly mixing probability-preserving system
$(X,\cB,\mu,T)$ and a Borel set $A\subset X$ such that
\[
 \mu(A)>0,\qquad A\cap T^{-r}A=\varnothing
 \quad\bigl(r\in(S_{\cP}+h)\cap\N\bigr).
\]
The system and the set may depend on $h$.
\end{corollary}

\begin{proof}
	For $h=0$, the conclusion follows from
	Theorem~\ref{thm:main}. Suppose $h\ne0$.
	
	By Lemma~\ref{lem:valuation-cover}, write
$S_{\cP}=F\cup E_1\cup\cdots\cup E_s,$
	where $F$ is finite and each $E_i$ is
	infinite and satisfies
	\[
	\lim_{N\to\infty}
	\gcd\{r\in E_i:r\ge N\}=\infty.
	\]
	
	For each $1\le i\le s$, choose integers
	$c_i>|h|$ and $L_i>|h|$ such that
	\[
	c_i\mid r
	\qquad(r\in E_i,\ r\ge L_i).
	\]
	Put
$E_i'=E_i\cap[L_i,\infty).$
	Each $E_i'$ is infinite and still satisfies
	the same condition on greatest common
	divisors.
	Apply Lemma~\ref{lem:measures} to $E_i'$
	with $\eta=1$. We obtain
$E_i'=C_{i,0}\sqcup C_{i,1}$
	and unit complex numbers
	$\lambda_{i,\epsilon,k}\ne1$ such that
	\[
	\lambda_{i,\epsilon,k}\to1,
	\qquad
	\sup_{r\in C_{i,\epsilon}}
	|\lambda_{i,\epsilon,k}^r-1|
	\to0
	\]
	as $k\to\infty$, for $\epsilon=0,1$.
	
	Fix $i$ and $\epsilon$, and define
$Q_{i,\epsilon}
	=(C_{i,\epsilon}+h)\cup\{1\}.$
	Since $L_i>|h|$, all elements of
	$C_{i,\epsilon}+h$ are positive.
	Moreover, $Q_{i,\epsilon}$ is infinite
	and contains $1$.
	For every $r\in C_{i,\epsilon}$,
	\[
	\begin{aligned}
		|\lambda_{i,\epsilon,k}^{r+h}-1|
		&\le
		|\lambda_{i,\epsilon,k}^r-1|
		+|\lambda_{i,\epsilon,k}^h-1|\\
		&\le
		|\lambda_{i,\epsilon,k}^r-1|
		+|h|\,|\lambda_{i,\epsilon,k}-1|.
	\end{aligned}
	\]
	Consequently,
	\[
	\begin{aligned}
		\sup_{n\in Q_{i,\epsilon}}
		|\lambda_{i,\epsilon,k}^n-1|
	\le
		\sup_{r\in C_{i,\epsilon}}
		|\lambda_{i,\epsilon,k}^r-1|
		+(|h|+1)|\lambda_{i,\epsilon,k}-1|
		\longrightarrow0.
	\end{aligned}
	\]
	
	Next, let $\zeta_i=\ee{1/c_i}$.
	Since $c_i\mid r$ for $r\in C_{i,\epsilon}$,
$\zeta_i^{r+h}=\ee{h/c_i}.$
	As $0<|h|<c_i$, we have
	\[
	\begin{aligned}
		\inf_{n\in Q_{i,\epsilon}}
		|\zeta_i^n-1|
	\ge
		\min\left\{
		|\ee{h/c_i}-1|,
		|\ee{1/c_i}-1|
		\right\}>0.
	\end{aligned}
	\]
	Thus $Q_{i,\epsilon}$ satisfies the
	assumptions of Lemma~\ref{lem:grivaux}.
	We obtain an invertible weakly mixing
	system and a positive-measure Borel set
	with no returns at times in
	$Q_{i,\epsilon}$.
	
	The remaining positive times belong to
	the finite set
	\[
	H=
	\left(
	\left(
	F\cup\bigcup_{i=1}^s(E_i\setminus E_i')
	\right)+h
	\right)\cap\N.
	\]
	If $H$ is nonempty, apply
	Lemma~\ref{lem:finite-times} to $H$.
	Since
	\[
	(S_{\cP}+h)\cap\N
	\subset
	H\cup\bigcup_{i=1}^s
	\bigcup_{\epsilon=0}^1 Q_{i,\epsilon},
	\]
	Lemma~\ref{lem:product} gives the
	required system and Borel set.
\end{proof}

\begin{remark}\label{rem:general-translates}
	The proof of Corollary~\ref{cor:translates}
	also gives a more general result.
	Suppose
	\[
	R=F\cup E_1\cup\cdots\cup E_s\subset\N,
	\]
	where $F$ is finite, each $E_i$ is infinite,
	and
	\[
	\lim_{N\to\infty}
	\gcd\{r\in E_i:r\ge N\}=\infty
	\qquad(1\le i\le s).
	\]
	Then, for every $h\in\Z\setminus\{0\}$,
	there exist an invertible weakly mixing system
	$(X,\cB,\mu,T)$ and a Borel set $A$ such that
	\[
	\mu(A)>0,\qquad
	A\cap T^{-r}A=\varnothing
	\quad(r\in(R+h)\cap\N).
	\]
	
	Indeed, for each $i$, we can choose
	$c_i>|h|$ and $L_i>|h|$ such that
	\[
	c_i\mid r
	\qquad(r\in E_i,\ r\ge L_i).
	\]
	Hence
	\[
	|\ee{(r+h)/c_i}-1|
	=|\ee{h/c_i}-1|>0.
	\]
The argument used to prove Corollary~\ref{cor:translates}
	also applies to these sets.
We use an additional Gaussian shift to rule out returns
at the remaining finitely many times, and then take
the product of all the systems.
\end{remark}

\subsection{A compact symbolic realization}\label{sec:symbolic}
A homeomorphism $S$ of a compact metric space $Y$ is topologically
transitive if for any nonempty open sets $U,V\subset Y$, some
$n\geq0$ satisfies $U\cap S^{-n}V\ne\varnothing$. It is
topologically weakly mixing if $S\times S$ is topologically
transitive on $Y\times Y$. A set is clopen if it is both open
and closed.
Theorem~\ref{thm:main} also gives the following
result for symbolic systems.

\begin{proposition}\label{prop:symbolic}
Let $\cP=\{p_1,\ldots,p_d\}$ be a finite nonempty set of primes, and put
\[
 S_{\cP}=\{p_1^{n_1}\cdots p_d^{n_d}:n_1,\ldots,n_d\in\Nzero\}.
\]
There exist a compact set $Y\subset\{0,1\}^{\Z}$ with $S(Y)=Y$
for the shift $(Sy)_n=y_{n+1}$, and a nonempty clopen set $U\subset Y$,
such that
\[
 U\cap S^{-r}U=\varnothing\qquad(r\in S_{\cP}).
\]
The shift on $Y$ is topologically weakly mixing and admits a
full-support weakly mixing invariant Borel probability measure.
\end{proposition}

\begin{proof}
Take $(X,\cB,\mu,T)$ and $A$ from Theorem~\ref{thm:main}. Define
\[
 \Phi(x)_n=1_A(T^nx),\qquad m=\Phi_*\mu,\qquad Y=\supp m.
\]
Since $\Phi\circ T=S\circ\Phi$ and $T_*\mu=\mu$,
we have $S_*m=m$ and $S(Y)=Y$.
Moreover, applying \eqref{eq:wm} to
$f\circ\Phi$ and $g\circ\Phi$ shows that
$(Y,m,S)$ is weakly mixing.
For the clopen set $U=\{y\in Y:y_0=1\}$,
\[
 \begin{aligned}
 m(U)&=\mu(A)>0,\\
 m(U\cap S^{-r}U)&=\mu(A\cap T^{-r}A)=0
                   &&(r\in S_{\cP}).
 \end{aligned}
\]
Each $U\cap S^{-r}U$ is open in $Y$, so full support makes it empty.
To check topological weak mixing, take nonempty open sets
$U_1,U_2,V_1,V_2\subset Y$. Since $m$ has full support, we have
$c_i:=m(U_i)m(V_i)>0$.
By \eqref{eq:wm},
\[
\begin{aligned}
	&\frac1N\#\{1\le n\le N:
	m(U_i\cap S^{-n}V_i)=0\}\\
	&\quad\le
	\frac1{c_iN}\sum_{n=1}^N
	\left|m(U_i\cap S^{-n}V_i)-c_i\right|
	\longrightarrow0.
\end{aligned}
\] There is therefore
a common $n$ for which both intersections have positive measure.
Since open rectangles form a base for $Y\times Y$, this proves
topological transitivity of $S\times S$.
\end{proof}

The same construction applies to
Theorem~\ref{thm:general},
Corollaries~\ref{cor:infinite}
and~\ref{cor:translates}, and
Remark~\ref{rem:general-translates}.

	\section*{Acknowledgments}
The author would like to thank his advisor, Professor Siming Tu, for his continued guidance, support, and encouragement.
The author is also grateful to Professor Bryna Kra for sharing her earlier
unpublished notes and for drawing attention to
related results in the literature.

\smallskip

\emph{Statement on the Use of AI Tools in This Work:}
While preparing this work, the author used ChatGPT 6 for assistance with mathematical discussions and editorial revisions. All mathematical results, proofs, and interpretations were verified and approved by the
author, who takes full responsibility for the content of this work.

\end{document}